\documentclass[reqno]{amsart}
\usepackage{graphicx}
\usepackage{amssymb}
\usepackage{amsfonts}
\usepackage{indentfirst}
\usepackage[colorlinks,linkcolor=black,anchorcolor=blue,citecolor=blue]{hyperref}
\usepackage{amssymb,amsmath,amsthm}
\usepackage{latexsym,bm}
\usepackage{graphicx}
\usepackage{times}
\usepackage{indentfirst}

\def\Z{\mathbb Z}
\def\N{\mathbb N}
\def\Q{\mathbb Q}
\def\K{\mathbb K}
\def\O{{\mathcal O}_K}
\def\P{{\mathcal P}_K}
\def\U{\mathbb U}
\def\V{\mathbb V}
\def\X{\mathbb X}
\def\B{\mathbb B}

\def\a{\alpha}
\def\b{\beta}
\def\val{v_{\pi}}
\def\sig{\sigma}
\def\tsig{\tilde{{\sigma}}}

\def\Qp{\mathbb{Q}_{p}}

\def\Z{\mathbb{Z}}
\def\N{\mathbb{N}}

\def\Q{\mathbb{Q}}

\def\K{\mathbb K}

\def\O{{\mathcal O}_K}

\def\U{\mathbb U}
\def\V{\mathbb V}
\def\X{\mathbb X}

\def\a{\alpha}
\def\b{\beta}
\def\val{v_{\pi}}
\def\sig{\sigma}
\def\tsig{\tilde{{\sigma}}}

\theoremstyle{plain}
  \newtheorem{thm}{Theorem}[section]
  \newtheorem{prop}[thm]{Proposition}
  \newtheorem{lem}[thm]{Lemma}
  \newtheorem{cor}[thm]{Corollary}
  
  \newtheorem*{cth*}{Theorem}
  
  \newtheorem*{clm*}{Lemma}
  
\theoremstyle{definition}
  
  \newtheorem*{defn*}{Definition}
  
\theoremstyle{remark}
  \newtheorem{rem}{Remark}

\begin{document}
\title[]{Dynamics of convergent power series on the integral ring of a finite extension of $\mathbb{Q}_p$ }
\date{\today}
\author{Shilei Fan}
\address{School of Mathematics and Statistics, Central China Normal University, 430079, Wuhan, China}
\email{slfan@mail.ccnu.edu.cn}

\author{Lingmin Liao}
\address{LAMA, UMR 8050, CNRS,
Universit\'e Paris-Est Cr\'eteil, 61 Avenue du
G\'en\'eral de Gaulle, 94010 Cr\'eteil Cedex, France}
\email{lingmin.liao@u-pec.fr}

\begin{abstract}
Let $K$ be a finite extension of the field $\mathbb{Q}_p$ of $p$-adic numbers
 and $\O$ be its integral ring.  The convergent power series with coefficients in $\O$ are studied  as dynamical
systems on $\O$.  A minimal decomposition theorem  for such a dynamical system is obtained. It is proved that there are uncountably many minimal subsystems, provided that there is a
 minimal set consisting of infinitely many points.
 In particular, the complete detailed minimal decompositions of all affine systems are derived.
 \end{abstract}
\maketitle
\section{Introduction}
This paper contributes to the theory of $p$-adic dynamical systems which has recently been intensively developed.
We refer to Silverman's book \cite{SilvermanGTM241} and Anashin and Khrennikov's book \cite{AK09AAD} for this development.

Let $p$ be a prime number. Let $\mathbb{Q}_p$ be the field of $p$-adic numbers and $|\cdot|_p$ be the $p$-adic absolute value on $\mathbb{Q}_p$.
Consider a finite extension $K$ of $\mathbb{Q}_p$ with $d=[K:\Q_{p}]$  being the degree of the extension.
The extended absolute value of $K$ is still denoted by $|\cdot|_p$.
Let $\mathcal{O}_K:=\{x\in K : |x|_p \leq 1\} $ be the local ring of
$K$.  Define by
$$\mathcal{O}_K\langle x\rangle:=\left\{\sum_{i\geq 0}a_{i}x^i\in \mathcal{O}_K[[x]]: \lim_{i\rightarrow \infty}a_{i}=0  \right \}$$
 the class of convergent power series with coefficients in $\mathcal{O}_K$. Remark that when $K=\mathbb{Q}_p$, the class $\mathcal{O}_K\langle x\rangle$ was called the class $\mathcal{C}$  by Anashin in \cite{Ana02}.
 A power series $\phi\in \mathcal {O}_K\langle x\rangle$ is considered  as a topological dynamical system on $\mathcal {O}_K$, denoted as $(\mathcal {O}_K, \phi)$.

Let $X$ be a compact metric space and $T$ be a continuous map from $X$ to itself.
Denote by $T^{\circ n}$ the $n$th  iterate of $T$, i.e., $$T^{\circ n}=\underbrace{T\circ T\circ\cdots \circ T}_{n \mbox{~times}}.$$
For $x\in X$, the forward orbit of $x$ is the set
$\{T^{\circ n}(x):n\geq 0\}$.
The system $(X,T)$ is said to be \emph{minimal}  if each orbit is dense in $X$.

In the literature, the minimality, or equivalently the
ergodicity with respect to the Haar measure, of polynomial
dynamical systems on the ring $\mathbb{Z}_p$ of $p$-adic integers
(which is the local ring of $\mathbb{Q}_p$) are extensively studied (\cite{Ana94,Ana02, Ana06, AKY11,CFF09,DP09,FLYZ07,FL11, GKL01, Jeong2013,Larin02}).
See also the monographs \cite{AK09AAD,KN04book}
and the bibliographies therein. In \cite{FL11}, Fan and Liao proved the following result
which asserts that a polynomial with coefficients in
$\mathbb{Z}_p$ and with degree at least $2$, as a dynamical
system on $\mathbb{Z}_p$, admits at most countably many minimal
subsystems.
\smallskip

\noindent{\bf Theorem A.} {\em  Let $\phi$ be a polynomial with coefficients in $\mathbb{Z}_p$ and with degree at least $2$. Then we can decompose the space $\mathbb{Z}_p$ as
$$
     \mathbb{Z}_p = A \bigsqcup B \bigsqcup C
$$
where $A$ is the finite set consisting of all periodic points of
$\phi$, $B= \bigsqcup_i B_i$ is the union of all (at most countably
many) clopen invariant sets such that each subsystem $\phi: B_i \to
B_i$ is minimal, and each point in $C$ lies in the attracting basin
of a periodic orbit or of a minimal subsystem. }
\medskip

As for the affine polynomials, Fan, Li, Yao and Zhou
\cite{FLYZ07} showed that apart from some special cases like $ax$ with $a^n=1$, the systems have similar dynamical structures.

The decomposition in Theorem A is referred to as a minimal
decomposition. The topological dynamical structure of a polynomial
is hence totally described by its minimal decomposition. There are few
works on the minimal decomposition. Multiplications $f(x)=ax$, with $|a|_p=1$ on
$\mathbb{Z}_p$ ($p\ge 3$) were studied by Coelho and Parry
\cite{CP11Ergodic}.  All the
minimal components of a general affine polynomial on $\mathbb{Z}_p$ were exhibited by Fan, Li, Yao and Zhou \cite{FLYZ07}.  The minimal decomposition for
the quadratic polynomials in the case $p=2$ was obtained by Fan and Liao
\cite{FL11}. The minimal decomposition for
homographic maps on  the projective line over $\Qp$ was studied by Fan, Fan, Liao and Wang \cite{FFLW2013}.

In the present paper, we would like to study the dynamical systems $(\mathcal {O}_K, \phi)$ with $\phi\in \mathcal {O}_K\langle x\rangle$ being a convergent series. As we will show, the results in the
finite extension cases are different. If $K$ is a finite extension of $\Q_p$ and $K\neq \Q_{p}$, then $(\mathcal {O}_K, \phi)$ are always  non-minimal. Moreover, there are  uncountably many minimal subsystems of $(\mathcal {O}_K, \phi)$ for each $\phi\in \mathcal {O}_K\langle x\rangle $,  if there is a non-trivial minimal subset (i.e. a minimal subset consisting of infinitely many points).

Let $e$ be the ramification index of $K$ over $\Q_{p}$. We will prove the following theorem. The reader is referred to Section \ref{preliminary} for the definition of type $(k,e)$ of a subsystem.  Actually, a type $(k,e)$ subsystem contains uncountably many minimal subsystems which have the same dynamical structure (in fact, they are all conjugate to the same adding machine, see Section \ref{preliminary}). 


\begin{thm}\label{thm-decomposition}
 Let $\phi \in \mathcal{O}_{K}\langle x\rangle$. Suppose that  for each $n\geq 1$, $\phi^{\circ n}$ is not identity.
Then we have the following decomposition
$$
    \mathcal{O}_{K} = A \bigsqcup B \bigsqcup C
$$
where $A$ is the finite set consisting of all periodic points of
$\phi$,
$B= \bigsqcup_i B_i$ is the union of all (finite or countably
many) clopen invariant sets such that each $B_i$ is a finite union
of balls and each subsystem $\phi: B_i \to B_i$ is of type $(k\ell p^n,e)$ for some positive integers $1\leq k \leq p^f$, $\ell |(p^f-1)$, and $n\in \mathbb{N}$.
Each
point in $C$ lies in the attracting basin of a periodic orbit or of
a  subsystem of type $(k\ell p^n,e)$.

Moreover,  there are uncountably many minimal subsystems if $K\neq \mathbb{Q}_p$.
\end{thm}

To end this section, we point out that the obtained results could be applied to the study of rational maps on $\mathbb{Q}_p$. Since a rational map with coefficients in $\mathbb{Q}_p$ may have fixed points out of $\mathbb{Q}_p$, but in a finite extension of $\mathbb{Q}_p$. Thus one could study the corresponding dynamics on this finite extension, then translate the results to the restricted subsystem on $\mathbb{Q}_p$. A good example is the recent work \cite{FFLW2013} where multiplication dynamics on a quadratic extension is studied to obtain the dynamical structure of homographic dynamics on $\mathbb{Q}_p$.

The paper is organized as follows. In Section \ref{S-finite-ext}, some preliminaries on finite extensions of $\mathbb{Q}_p$ are recalled. In Section \ref{S-power-series}, we discuss some fundamental properties of the class $ \mathcal{O}_{K}\langle x\rangle$ of convergent series which will be useful for the proof of the main theorem. The techniques to study the minimality from local to global are fully developed in Section \ref{preliminary}. The proof of the main theorem will be given in Section \ref{S-proof}. At the end, as an application, we study the affine polynomials in Section \ref{S-affine}.
%

\bigskip
\section{Finite extensions of the field of $p$-adic numbers}\label{S-finite-ext}
We recall some basic notations and facts of finite extensions of the field of $p$-adic numbers.
The reader may consult \cite{Mah81,KobGTM58,Rob-GTM198,sch} for  more information on non-Archimedean fields.

Let $K$ be a finite
extension of the field $\Q_{p}$ of $p$-adic numbers. Denote by $d=[K:\Qp]$
the degree of the extension, i.e., the dimension of $K$ as a vector space over $\Qp$.
 The extended absolute value on $K$ is still denoted by $|\cdot|_{p}$.
For $x\in K^{*}:=K\setminus \{0\}$,  let $v_{p}(x):=-\log_{p}(|x|_{p})$
define the valuation of $x$, with convention $v_{p}(0):=\infty$.
One can show that there exists a unique positive integer $e$ which
is called the \emph{ramification index} of $K$ over $\Qp$, such that
$$v_{p}(K^*)=\frac{1}{e}\Z.$$ The extension $K$ over $\Qp$ is said to
be \emph{unramified} if $e=1$, \emph{ramified} if $e>1$ and
\emph{totally ramified} if $e=d$.  An element $\pi\in K$ is called an
\emph{uniformizer} if $v_{p}(\pi)=1/e$. The {local ring} $\mathcal{O}_{K}$ of $K$ is $\{x\in
K : |x|_p\leq 1\}$, whose elements are called  \emph{integers} of $K$.
The maximal idea of $\O$ is  $\mathcal{P}_{K}:=\{x\in K: |x|_{p}<1\}$.
Denote by $\mathbb{K}$ the residual class field $\O/\mathcal{P}_{K}$ of $K$.
 Then
$\mathbb{K}=\mathbb{F}_{p^f}$,  the finite field of $p^{f}$ elements
where $f=d/e$. Let $C=\{c_{0},c_{1},\cdots,c_{p^f-1}\}$ be a fixed
complete set of representatives of the cosets of $\mathcal{P}_{K}$
in $\mathcal{O}_{K}$. Then for each uniformizer $\pi$, every $x\in K$ has a unique $\pi$-adic
expansion of the form
\begin{equation}\label{piadic}
   x=\sum_{i=i_{0}}^{\infty}a_i\pi^{i}
\end{equation}
where $i_{0}\in \Z$ and $a_i\in C$ for all $i\geq i_{0}$.
For convenience of notations, we define $\val(x):=e\cdot v_{p}(x)$
for $x\in K $. Then $\val(K^*)=\Z $.

\bigskip
\section{Power series on $\O$}\label{S-power-series}
Recall that  $\mathcal{O}_K\langle x\rangle$ is the class of convergent power series with coefficients in $\mathcal{O}_K$. In this section, we mainly discuss about the stability of $\mathcal{O}_K\langle x\rangle$ under the classical operations.

It is easy to check that $\mathcal{O}_K\langle x\rangle$ is an algebra with respect to addition and multiplication.
\begin{prop}
The class $\mathcal{O}_K\langle x\rangle$  is closed with respect to  derivation and composition of functions.
\end{prop}
\begin{proof}
Let $$\phi(x)=\sum_{i=0}^{\infty}a_i x^i \in \mathcal{O}_K\langle x\rangle.$$
By noticing the facts that $\lim\limits_{i\rightarrow \infty} ia_i=0$ and
$\phi^{\prime}(x)=\sum_{i=1}^{\infty}i a_{i}x^{i-1},$
we have $\phi' \in \mathcal{O}_K\langle x\rangle$. Thus $\mathcal{O}_K\langle x\rangle$  is closed with respect to  derivations.

Now we prove that $\mathcal{O}_K\langle x\rangle$  is closed with respect to compositions of functions.  Let $\phi(x)=\sum_{i=0}^{\infty}a_i x^i, \psi(x)= \sum_{j=0}^{\infty}b_j x^j\in \mathcal{O}_K\langle x\rangle$. Notice that
$$\phi\circ \psi(x)=a_0+\sum_{k=1}^{\infty} a_{k}b_{0}^k +\sum_{i= 1}^{\infty}\left(\sum_{k=1}^{\infty}a_k
\sum_{\begin{subarray}{l}
    i_1+i_2+\cdots +i_k=i    \\
    i_1,\cdots, i_{k}\geq 0 \\
    \end{subarray} }b_{i_1}b_{i_2}\cdots b_{i_k}\right)x^i.$$
Since $\lim\limits_{i\rightarrow \infty}a_{i}=0$ and $\lim\limits_{j\rightarrow \infty}b_{j}=0$, the strong triangle inequality implies
$$\lim_{i\to \infty} \sum_{k=1}^{\infty}a_k
\sum_{\begin{subarray}{l}
    i_1+i_2+\cdots +i_k=i    \\
    i_1,\cdots, i_{k}\geq 0 \\
    \end{subarray} }b_{i_1}b_{i_2}\cdots b_{i_k} =0.$$
Thus $\phi\circ \psi \in \mathcal{O}_K\langle x\rangle.$
\end{proof}
\begin{rem} Remind that when $K=\mathbb{Q}_p$, the class $\mathcal{O}_K\langle x\rangle$ was called the class $\mathcal{C}$ in \cite{Ana02} where the stability of derivation for $\mathcal{C}$ was also proved. However, the stability of composition which was not checked in \cite{Ana02} will be needed in the proof of Theorem \ref{thm-decomposition}. We point out that the stability of composition is different from the classical stability of composition of convergent series in the field $K$. Here, we have to verify the conditions on the coefficients of a convergent series.
\end{rem}


For a power series $\phi(x)=\sum_{i=0}^{\infty}a_i x^i \in\mathcal{O}_K\langle x\rangle$, the largest integer $j$
such that $|a_j|_p=1$ is called the \emph{Weierstrass degree} of $\phi$, denoted by $wideg(\phi)$. If all coefficients of $\phi$ are in $\mathcal{P}_{K}$, we
 say that the Weierstrass degree of $\phi$ is infinite.
We will use the following Weierstrass Preparation Theorem
(see Sections 5.2.1-5.2.2 of \cite{BGR}).

\begin{thm}[Weierstrass Preparation Theorem]\label{WPT}
Let $\phi(x)=\sum_{i=0}^{\infty}a_i x^i \in\mathcal{O}_K\langle x\rangle$ be a nonzero convergent series with $wideg(\phi)<\infty$.  Let $j$ be the largest integer such
that $|a_j|_p=\max_{i\geq 0}|a_i|_p$. Then there is a monic polynomial $g\in \O[x]$ of degree $j$ and a power series $h\in\mathcal{O}_K\langle x\rangle$ such that $\phi=gh$, and $h(x)\neq 0 $ for all $x\in \O$.
\end{thm}

\bigskip
\section{Induced dynamics on $\O/{\pi}^n \O$} \label{preliminary}

This section is devoted to local dynamics on $\O/{\pi}^n \O$ which will deduce the global dynamics on $\O$.

In the remainder of this paper we require $\pi$ to be a fixed uniformizer of $K$ and  $C=\{c_{0},c_{1},\cdots,c_{p^f-1}\}$ be a fixed
complete set of representatives of the cosets of $\mathcal{P}_{K}$ in $\mathcal{O}_{K}$. Then every
$x\in \mathcal{O}_{K}$ has a unique $\pi$-adic
expansion of the form
\begin{equation*}
   x=\sum_{i=0}^{\infty}a_i\pi^{i}
\end{equation*}
with $a_i\in C$ for all $i\geq 0$. We remark that the results in this paper do not depend on  the choices of $\pi$ and $C$.

Let $\phi \in \mathcal{O}_{K}\langle x\rangle$ be a convergent power series  of integral
coefficients. The dynamics of $(\mathcal{O}_{K}, \phi)$ can be
derived by those of its induced finite dynamics on
$\mathcal{O}_{K}/\pi^n\mathcal{O}_{K}$.
Let $n\ge 1$ be a positive integer. Denote by $\phi_n$ the induced
mapping of $\phi$ on $\O/{\pi}^n \O$, i.e.,
$$\phi_n (x \ {\rm mod} \ {\pi}^{n})=\phi (x)  \ \quad \ \ {\rm mod} \ {\pi}^{n}.$$


\begin{thm}[\cite{Ana06,CFF09}]\label{minimal-part-to-whole}
Let $\phi\in \O\langle x\rangle$ and $E\subset \O$ be a compact $\phi$-invariant set.
Then $\phi: E\to E$ is minimal if and only if $\phi_n: E/\pi^n\O \to
E/\pi^n\O$ is minimal for each $n\ge 1$.
\end{thm}

By Theorem \ref{minimal-part-to-whole}, to study the minimality of $\phi$, it suffices to study the minimality of all finite dynamics $\phi_n$. To this end, we need an induction from the level $n$ to the level $n+1$. In other words, we wants to predict the dynamical structure of $\phi_{n+1}$ at level $n+1$ from that of $\phi_n$ at level $n$.
The idea of predicting comes from the work of
desJardins and Zieve \cite{DZunpu} where the case
$K=\Q_p$ ($p\geq 3$) was studied. This idea allowed Fan and Liao \cite{FL11} to give a
minimal decomposition theorem (Theorem A in the first section) for any polynomial with coefficients in $\mathbb{Z}_p$.

A collection $\sigma=(x_1, \cdots, x_k)$ of $k$ distinct points in $\O/{\pi}^n \O $ is called a
{\it cycle of $\phi_n$ of length $k$} or a {\it $k$-cycle at level $n$}, if
$$ \phi_n(x_1)=x_2, \cdots, \phi_n(x_i)=x_{i+1}, \cdots, \phi_n(x_k)=x_1.$$

Set
$$ X_\sigma:=\bigsqcup_{i=1}^k  X_i \ \ \mbox{\rm where}\ \
X_i:=\{ x_i +\pi^nt; \ t\in C\} \subset \O/{\pi}^{n+1} \O.$$
Then
\[
\phi_{n+1}(X_i) \subset X_{i+1}  \ (1\leq i \leq k-1) \ \ \mbox{\rm
and}\ \  \phi_{n+1}(X_k) \subset X_1.\]


In the following we shall study the behavior of the finite dynamics
$\phi_{n+1}$ on the $\phi_{n+1}$-invariant set $X_\sigma$ and determine all
cycles in $X_\sigma$ of $\phi_{n+1}$, which will be  called  {\it lifts} of
$\sigma$. Remark that the length of any lift $\tilde{\sigma}$ of
$\sigma$ is a multiple of $k$.

  Let $\mathbb{X}_i:=x_i+\pi^n \O=\{x\in \O: x\equiv x_i \ ({\rm mod} \ \pi^n)\}$ be the closed disk of radius $p^{-n/e}$  centered at $x_i\in \sigma$ and
  $$\mathbb{X}_{\sigma}:=\bigsqcup_{i=1}^{k} \mathbb{X}_i$$ be the clopen set corresponding to the cycle $\sigma$.

 Let $\psi:=\phi^{\circ k}$ be the $k$-th iterate of $\phi$. Then any point in $\sigma$
 is fixed by $\psi_n$, the $n$-th induced map of $\psi$. For any point $x\in \mathbb{X}_{\sigma}$,
 we have $(\psi(x)-x)/\pi^{n}\in \mathcal{O}_{K}$.
 Using Taylor Expansion, for $n\geq 1$, we have
 $$\psi(x+\pi^{n}t)\in x+\pi^{n}(\frac{\psi(x)-x}{\pi^{n}})+\pi^{n}\psi^{\prime}(x)t+\pi^{2n}\O, \quad \forall x\in \mathbb{X}_{\sigma}, \ \forall t\in\O.$$
Then it leads us to define the following two functions from $\mathbb{X}_{\sigma}$ to $\O$.
 For $x\in \mathbb{X}_{\sigma}$, let
\begin{eqnarray}
& &a_n(x):=\psi'(x)=\prod_{j=0}^{k-1} \phi'(\phi^{\circ j}(x)) \\
& &b_n(x):=\frac{\psi(x)-x}{\pi^n}=\frac{\phi^{\circ k}(x)-x}{\pi^n}.
\end{eqnarray}
The following lemma shows that the function $a_n(x)$
(mod $\pi^{n}$) is always constant on $\mathbb{X}_\sigma$.

\begin{lem}\label{lem1}
Let $n\geq 1 $ and $\sigma=(x_1, \cdots, x_k)$ be a $k$-cycle of $\phi_n$. \\
\indent {\rm (i)}\ For  $1\leq i\leq k$, we have $$
a_n(x+t\pi^n)\equiv a_n(x) \ \ ({\rm mod} \ \pi^n), \ \  \forall x\in \mathbb{X}_i, \ \forall t\in \O.$$

 \indent {\rm (ii)}\ For $1\leq i,j\leq k$, we have
$$a_n(x)\equiv a_n(y) \ \ ({\rm mod} \ \pi^n), \ \ \forall x\in \mathbb{X}_i, \ \forall y \in\mathbb{X}_j.$$

\indent{\rm (iii)}\ For $1\leq i\leq k$, we have
\[ b_n(x) \equiv b_n(x+t \pi^{n}) \quad ({\rm mod} \ \pi^A), \ \ \forall x\in \mathbb{X}_i, \ \forall t\in \O.
\]
where
 $A:=\min \{ v_{\pi} (a_n(x)-1),n \}$ for $x\in \mathbb{\sigma}$.
\\
 \indent {\rm (iv)}\  If $a_n(x) \not\equiv 0 \ ({\rm mod} \
{\pi})$ for some $x\in \mathbb{X}_\sigma$, then  we
have
$$\min \{v_{\pi}(b_n(x)),n \}=
\min \{v_{\pi}(b_n(y)),n \}, \ \forall x,y \in \mathbb{X}_{\sigma}. $$ Consequently, $\min
\{v_{\pi}(b_n(x)),A\}= \min \{v_{\pi}(b_n(y)),A \}$.
\end{lem}

\begin{proof}
Since  $\phi \in \O\langle x\rangle$, it follows that $\phi^{\prime}\in \O\langle x\rangle$.
 Then, the assertion (i) is a direct consequence of
\[
a_n(x+t\pi^n) \equiv  \prod_{j=0}^{k-1} \phi'(\phi^{\circ j} (x+t\pi^n)) \equiv
\prod_{j=0}^{k-1} \phi'(\phi^{\circ j}(x))  \quad ({\rm mod} \ {\pi}^n),
\]
where the second equality follows form that $\phi^{\prime}\in\O\langle x\rangle$.

 Assertion {\rm (ii)} follows directly from the definition of
$a_n(x)$ and the facts that $\phi^{\prime}\in\O\langle x\rangle$ and  $\sigma = (x_i, \phi_n(x_i), \cdots,
\phi_n^{\circ k-1}(x_i))$.

The $1$-order Taylor Expansion of $\psi$ at $x$ gives
\begin{align*}
\psi(x+{\pi}^n t) -(x+{\pi}^n t) 
 \equiv {\pi}^n\left(\frac{\psi(x)-x}{{\pi}^n}\right) +{\pi}^n t(\psi'(x)-1) \quad ({\rm mod} \ {\pi}^{2n}).
\end{align*}
Hence
\begin{align}\label{bnne}
 b_{n}(x+{\pi}^n t)\equiv b_n(x) + t (a_n(x)-1) \quad ({\rm mod} \ {\pi}^{n}).
 \end{align}
Then (iii) follows.

Write
$$
\psi(\phi(x))-\phi(x) = \phi(\phi^{\circ k}(x))-\phi(x) = \phi(x+{\pi}^n b_n(x))
-\phi(x).
$$
The $1$-order  Taylor Expansion of $\phi$ at $x$ implies $$
\psi(\phi(x))-\phi(x) \equiv {\pi}^n b_n(x) \phi'(x) \quad ({\rm mod} \
{\pi}^{2n}).
$$
Hence we have
\[b_{n}(\phi(x))\equiv b_n(x) \phi'(x) \quad ({\rm mod} \ {\pi}^{n}).
\]
Thus we obtain (iv), because $a_n(x) \not\equiv 0 \ ({\rm mod} \
{\pi})$ (for some $x \in \mathbb{X}_\sigma $) implies $\phi'(y) \not\equiv 0 \
({\rm mod} \ {\pi})$ for all $y\in \mathbb{X}_\sigma$.
\end{proof}

\medskip

The $1$-order Taylor Expansion of $\psi$ at $x$ implies
\begin{eqnarray}\label{linearization}
 \psi(x+\pi^n t) \equiv x+\pi^n b_n(x) + \pi^n a_n(x) t \quad ({\rm mod} \ \pi^{2n}).
 \end{eqnarray}
 Define
an affine map  $$\Psi: \mathbb{X}_{\sigma}\times (\O/\pi\O) \rightarrow \O/\pi\O $$
by
$$\Psi(x,t)=b_n(x)+a_n(x) t \quad ({\rm mod} \ \pi).$$

An important consequence of  formula (\ref{linearization}) shows that $\psi_{n+1}:
X_i \to X_i$ is conjugate to the linear map $$ \Psi(x, \cdot):
\O/{\pi} \O \to \O/{\pi} \O$$ for some $x\in \mathbb{X}_i$.

We could call it  the \emph{linearization}
of $\psi_{n+1}: X_i \to X_i$.

In order to determine the linearization of $\psi_{n+1}$, we only need to
get the values modulo $\pi$ of $a_n(x)$ and $b_n(x)$. As we see in Lemma \ref{lem1}, the function $a_n(x)$
(mod $\pi$) is always constant on $\mathbb{X}_\sigma$ and the function $b_n(x)$
(mod $\pi$) is also constant on each $\mathbb{X}_i$ but under the condition
$a_n(x)\equiv 1$ (mod $\pi$).
For simplicity, sometimes we shall write $a_n$
and $b_n$ without mentioning $x$ when we only care about the value of $a_n(x)\ ({\rm mod} \ {\pi}^{n})$ and when we only care about $b_n(x)\equiv 0\ ({\rm mod} \ {\pi}) $ or $b_n(x)\not\equiv 0\ ({\rm mod} \ {\pi}) $ under $a_n(x)\equiv 1\ ({\rm mod} \ {\pi}^{n})$.


\medskip
We remind that the cardinality of the residual field $\K =\O/\P =\O/{\pi} \O $ of $K$ is
$p^f$. The characteristic of the field $\K $ is $p$. The multiplicative group $\K^* $ is cyclic of order $p^f-1$.

The above analysis allows us to
distinguish the following four behaviors  of $\phi_{n+1}$ on $X_{\sigma}$:\\
 \indent {\rm (a)} If $a_n \equiv 1 \ ({\rm mod} \ \pi)$ and
 $b_n \not\equiv 0 \ ({\rm mod} \ \pi)$, then for any $x\in \mathbb{X}_\sigma$, $\Psi(x, \cdot)$ preserves $p^{f-1}$ cycles of length $p$, so
 $\phi_{n+1}$ restricted to $X_\sigma$ preserves $p^{f-1}$ cycles of length $pk$. In this case we say $\sigma$ {\it grows}.\\
 \indent {\rm (b)} If $a_n \equiv 1 \ ({\rm mod} \ \pi)$ and
 $b_n \equiv 0 \ ({\rm mod} \ \pi)$, then for any $x\in \mathbb{X}_\sigma$, $\Psi(x, \cdot)$ is the identity, so $\phi_{n+1}$ restricted to $X_\sigma$
 preserves
 $p^f$ cycles  of length $k$. In this case we say $\sigma$ {\it splits}. \\
 \indent {\rm (c)} If $a_n \equiv 0 \ ({\rm mod} \ \pi)$, then for any $x\in \mathbb{X}_\sigma$, $\Psi(x, \cdot)$ is constant, so $\phi_{n+1}$ restricted to $X_\sigma$
 preserves
 one cycle of length $k$ and  the remaining points of $X_\sigma$ are mapped into this cycle.
  In this case we say $\sigma$ {\it grows tails}. \\
  \indent {\rm (d)} If $a_n \not\equiv0, 1 \ ({\rm mod} \ \pi)$, then for any $x\in \mathbb{X}_\sigma$, $\Psi(x, \cdot)$ is a permutation
  and the $l$-th iterate of $\Psi(x, \cdot)$ reads
 \begin{eqnarray*}
\Psi^{\circ l}(x, t) = b_n(a_n^{l}-1) /(a_n-1) +a_n^{l} t,
\end{eqnarray*}
so that $$ \Psi^{\circ l}(x,t)-t =(a_n^{l}-1) \left( t+
\frac{b_n}{a_n-1}\right).$$
 Thus,
$\Psi(x, \cdot)$ admits a single fixed point $t=-b_n/(a_n-1)$, and the
remaining points  lie on cycles of length $\ell$, where $\ell$ is the
order of $a_n$ in $(\O/ \pi\O)^*=\K^*$. So, $\phi_{n+1}$ restricted to
$X_\sigma$ preserves one cycle of length $k$ and $\frac{p^f-1}{\ell}$ cycles
of length $k\ell$. In this case we say $\sigma$ {\it partially splits}.
\medskip

Let $\sigma=(x_1,\dots,x_k)$ be a $k$-cycle of $\phi_n$ and let
$\tilde{\sigma}$ be a lift of $\sigma$ of length $kr$, where $r\geq
1$ is an integer. It follows immediately that $\mathbb{X}_{\tilde{\sigma}}\subset \mathbb{X}_\sigma$. For $x\in \mathbb{X}_{\tilde{\sigma}}$, we shall study the relation between $(a_n, b_n)$ and $(a_{n+1},
b_{n+1})$. Our aim is to see the change of nature from a cycle to
its lifts.

\begin{lem}\label{anbn} We have
\begin{eqnarray}\label{an}
a_{n+1}(x) \equiv a_n^r(x) \quad ({\rm mod} \ \pi^{n}),
\end{eqnarray}

\begin{equation}\label{bn}
\pi b_{n+1}(x) \equiv   b_n(x)(1+a_n(x)+ \cdots +a_n^{r-1}(x))  \quad ({\rm mod} \ \pi^{n}).
\end{equation}
\end{lem}

\begin{proof}
 The formula (\ref{an}) follows from
\[
a_{n+1}(x)  \equiv (\psi^{\circ r})'(x) \equiv
\prod_{j=0}^{r-1} \psi'(\psi^{\circ j}(x))
 \equiv a_n^r(x) \ ({\rm mod} \ \pi^{n}).
 \]

 By repeating $r$-times of the linearization (\ref{linearization}), we obtain
 $$
 \psi^{\circ r}(x) \equiv x + \Psi^{\circ r}(x, 0) \pi^n \quad ({\rm mod} \
 \pi^{2n}),
 $$
 where $\Psi^{\circ r}(x,t)$ means the $r$-th composition of $\Psi(x,r)$ as function
 of $t$. However,
$$
\Psi^{\circ r}(x ,0) = b_n(x)
 (1+a_n(x)+ \cdots + a_n^{r-1}(x)).
$$
Thus (\ref{bn}) follows from
  the definition of $b_{n+1}$ and the above two expressions.
\end{proof}

By Lemma \ref{anbn}, we immediately obtain the following
proposition.
\begin{prop}
Let $n\geq 1$. Let $\sigma$ be a $k$-cycle of $\phi_n$
and $\tilde{\sigma}$ be a lift of $\sigma$. Then  we have\\
 \indent {\rm 1)} if $a_n \equiv 1 \ ({\rm mod} \ \pi)$, then $a_{n+1} \equiv 1 \ ({\rm mod} \ \pi)$;\\
 \indent {\rm 2)} if $a_n \equiv 0 \ ({\rm mod} \ \pi)$, then $a_{n+1} \equiv 0 \ ({\rm mod} \ \pi)$;\\
 \indent {\rm 3)} if $a_n \not\equiv 0,1 \ ({\rm mod} \ \pi)$ and $\tilde{\sigma}$ is of length
 $k$, then $a_{n+1} \not\equiv 0,1 \ ({\rm mod} \ \pi)$;\\
 \indent {\rm 4)} if $a_n \not\equiv 0,1 \ ({\rm mod} \ \pi)$ and $\tilde{\sigma}$ is of length
 $k\ell$ where $\ell \ge 2$ is the order of $a_n$ in $(\O/\pi\O)^*$, then $a_{n+1} \equiv 1 \ ({\rm mod} \ \pi)$.
\end{prop}

This result is interpreted as follows in a dynamical way.\\
 \indent {\rm i)} If $\sigma$ grows or splits, then any lift $\tilde{\sigma}$ grows
or splits. \\
 \indent {\rm ii)}
If $\sigma$ grows tails, then the single lift $\tilde{\sigma}$
also grows tails. \\
 \indent {\rm iii)}
 If $\sigma$ partially splits, then the
 lift $\tilde{\sigma}$
 of the same length as $\sigma$ partially splits, and the other lifts  of length $k\ell$ grow or split.

\medskip
By the above discussion, the case of growing tails is simple.
If $\sigma=(x_1, \cdots, x_k)$ is a cycle of $\phi_n$ which grows
tails, then $\phi$ admits a $k$-periodic point $x_0$ in the clopen set
$\mathbb{X}_\sigma$ and $\mathbb{X}_\sigma$ is
contained in the attracting basin of the periodic orbit $x_0,
\phi(x_0), \cdots, \phi^{\circ (k-1)}(x_0)$.

Similarly, for the case of partially splitting, we can also find a periodic orbit in $\mathbb{X}_\sigma$, and other parts are reduced to the cases of growing and splitting. Hence, we mainly study these later cases.

For a cycle $\sigma=(x_1,\dots,x_k)$ at level $n$ and $x\in \mathbb{X}_\sigma$, let $$A_n(x):= \val(a_n(x)-1), \quad B_n(x):= \val(b_n(x)).$$
By Lemma
\ref{lem1}, $\hat{A}_{n}(x):=\min\{A_{n}(x),n\}$ does not depend on the choice of $x\in \mathbb{X}_{\sigma} $ and if $B_n(x) < \min \{A_n(x), n \} $ then $B_n(x)$ does
not depend on the choice of $x\in \mathbb{X}_\sigma$. Sometimes, there
is no difference when we choose $x\in \mathbb{X}_\sigma$. So,
without misunderstanding, we shall write $A_n, \hat{A}_{n}$ and
$B_n$ without  mentioning $x$.

First, we study the splitting case. We say a cycle {\it splits $l$ times} if
itself splits, its lifts split, the second generation of the
descendants split, ... , and the ($l-1$)-th generation of the
descendants split.

\begin{prop}\label{splits-p3}
Let $\sigma$ be a splitting cycle of $\phi_n$. \\
 \indent {\rm  1)} If $\hat{A}_n > B_n $,  then every lift splits $B_n-1$
 times then all  lifts at level $n + B_n$ grow.\\
 \indent {\rm  2)} If $A_n\leq B_n$ and $A_n< n$,  then there is one lift which behaves the same as
 $\sigma$ (i.e., this lift splits and $A_{n+1}\leq B_{n+1}$, $A_{n+1}< {n+1}$) and other lifts split $A_n-1$ times then  all  lifts at level $n + A_n$ grow.\\
 \indent {\rm  3)} If $B_n \geq n$ and $A_n \geq n$, then all lifts split at least $n-1$
 times.
\end{prop}
\begin{proof}
Taking $r=1 $ in Lemma \ref{anbn}, we have
 \begin{eqnarray}\label{1}
 a_{n+1}(x)-1 \equiv a_n(x)-1  \quad ({\rm mod} \ \pi^{n})
\end{eqnarray}
and
\begin{eqnarray}\label{2}
   \pi b_{n+1}(x) \equiv b_n(x)  \quad ({\rm mod} \ \pi^{n}).
\end{eqnarray}
 \indent {\rm  1)} Since  $\hat{A}_n > B_n $, the assertion {\rm (ii)} of Lemma \ref{lem1}
 says that  $$b_{n}(x+{\pi}^n t)\equiv b_n(x)  \quad ({\rm mod} \ {\pi}^{\hat{A}_n})$$  for all $t\in \O$.
 Replacing $x$ with $x+{\pi}^n t$ in the equalities (\ref{1}) and (\ref{2}), we have $B_{n+1}(x)=B_{n}(x)-1$ and  $\hat{A}_{n+1}(x)\geq \hat{A}_{n}(x) \geq B_{n+1}(x)$ for all $x\in \X_\sigma$.
 So the lifts of $\sigma$ split. By induction, the lifts will split $B_{n}-1$ times and $B_{n+B_{n}}=0$. This means that $b_{n+B_{n}}(x)\neq 0  ~({\rm mod} \ \pi)$ for all $x\in \X_{\sigma}$. So all the descendants of $\sigma$ at level  $n + B_n$ grow.\\
  \indent {\rm  2)} Since  $A_n\leq B_n$ and $A_n< n$, the equality (\ref{bnne}) implies that there exists a lift $\hat{\sigma}$ such that  $B_n(x)>B_n(y)= A_n$
  for all $x\in \X_{\hat{\sigma}}$ and $y\in \X_{\sigma}\setminus \X_{\hat{\sigma}}$. By equalities (\ref{1}) and (\ref{2}), we have $B_{n+1}(x)\geq A_{n+1}(x)$ for all $x\in \X_{\hat{\sigma}} $ and $B_{n+1}(y)= A_{n}(y)-1 <A_{n+1}(y)$ for all $y\in \X_{\sigma}\setminus \X_{\hat{\sigma}}$. So the lift $\hat{\sigma}$  behaves the same as $\sigma$. By the proved assertion 1), the other lifts split $A_n-1$ times then  all  lifts at level $n + A_n$ grow. \\
  \indent {\rm  3)} Since   $B_n \geq n$ and $A_n \geq n$, the equality (\ref{bnne})
 implies  that  $$b_{n}(x)\equiv 0  \quad ({\rm mod} \ {\pi}^{n})$$  for all $x\in \X_{\sigma}$. By the equality (\ref{2}), $B_{n+1}(x)\geq n-1$, so all the lifts split. By induction, all the lifts split at least $n-1$ times.
\end{proof}


Now we study the growing case. We begin with a lemma and a proposition.
\begin{lem}\label{4.6}
Let $\sigma$ be a growing cycle and let  $\tsig $ be a lift of $\sigma$. If $n> \val(1 + a_n(x)+ \cdots + a_n^{p-1}(x))$ for some
$x\in \mathbb{X}_\sigma$, then $n> \val(1 + a_n(y)+ \cdots + a_n^{p-1}(y))$ for all $y\in \mathbb{X}_\sigma$.
\end{lem}
\begin{proof}
If $A_n(x)\geq n$, then $a_{n}(x) \equiv 1  \ ({\rm mod} \ \pi^{n})$.
The condition $n> \val(1 + a_n(x)+ \cdots + a_n^{p-1}(x))$ implies  $n>e$.
According to Lemma \ref{anbn}, we have $a_{n}(y) \equiv 1  \ ({\rm mod} \ \pi^{n})$ for all $y\in \mathbb{X}_\sigma$.
Thus $n>\val(1 + a_n(y)+ \cdots + a_n^{p-1}(y))=e$.

If $A_n(x)<n$, then the condition  $n> \val(1 + a_n(x)+ \cdots + a_n^{p-1}(x))$ implies that
$\val(1 + a_n(x)+ \cdots + a_n^{p-1}(x))=\val(1 + a_n(y)+ \cdots + a_n^{p-1}(y))$  for all $y\in \mathbb{X}_\sigma$. So $n>\val(1 + a_n(y)+ \cdots + a_n^{p-1}(y))$.

\end{proof}

\begin{prop}\label{grow-2}
 Let $\sigma$ be a growing cycle and  $\tsig $ be a lift of $\sigma$.
Then if $n> \val(1 + a_n+ \cdots + a_n^{p-1})$, we have
 $\tsig $ splits $\val(1 + a_n+ \cdots + a_n^{p-1})-1$ times then all the descendants grow.
 \end{prop}
 \begin{proof}
 Taking $r=p $ in Lemma \ref{anbn}, we have
 \begin{align*}
 a_{n+1}(x)-1 &\equiv a_n(x)^p-1  \quad ({\rm mod} \ \pi^{n}) \\
 &\equiv (a_n(x)-1)(1 + a_n(x)+ \cdots + a_n^{p-1}(x)) \quad ({\rm mod} \ \pi^{n}).
\end{align*}
So we have $A_{n+1}(x)\geq \min\{n,A_n(x)+\val(1 + a_n(x)+ \cdots + a_n^{p-1}(x))\}$.

Since $\sig $ grows, $b_{n}(x) \not\equiv 0\ ({\rm mod} \ \pi)
$. By Lemma \ref{anbn}, we have
\begin{align}\label{bn-rp}
   \pi b_{n+1}(x) &\equiv b_n(x)(1+a_n(x)+ \cdots +
a_n(x)^{p-1})  \quad ({\rm mod} \ \pi^{n}).
\end{align}
 Hence if $n> \val(1 + a_n(x)+ \cdots + a_n^{p-1}(x))$,
\[
B_{n+1}(x)=\val(b_{n+1}(x))=\val(1 + a_n(x)+ \cdots + a_n^{p-1}(x))-1.
\]

Therefore, by the assertion 1) of Proposition  \ref{splits-p3}, $\tsig $ splits $\val(1 + a_n+ \cdots + a_n^{p-1})-1$ times then all the descendants grow.
 \end{proof}

Lemma \ref{4.6} and Proposition \ref{grow-2} lead to the following proposition.
\begin{prop}\label{grow-3}
  Suppose  $ n\geq e+1$. If $\sigma$ grows and $A_n(x)>e/(p-1)$ for some $x\in \sigma$, then all the $p^{f-1}$ lifts of $\sigma$ split $e-1 $ times and then all the descendants at level $n+e$ grow.
\end{prop}
\begin{proof} Since $\sigma$ grows and $A_n(x)>e/(p-1)$, we can distinguish two cases : $a_n\equiv 1 \ ({\rm mod} \ \pi^{n})  $ and $a_n=1+\pi^{\gamma}\delta$ with $|\delta|_p=1$ and  $e/(p-1)<\gamma<n$.

If $a_n\equiv 1 \ ({\rm mod} \ \pi^{n})  $, then by $ n\geq e+1$, we have  $ 1+a_n+\cdots+a_n^{p-1}\equiv p \ ({\rm mod} \ \pi^{n})$.

If $a_n=1+\pi^{\gamma}\delta$ with $|\delta|_p=1$ and  $e/(p-1)<\gamma<n$, then we have
\begin{align*}
1+a_n+\cdots + a_n^{p-1} = \frac{a_n^p-1}{a_n-1}&=\frac{{p\choose
1}\pi^{\gamma}\delta+{p\choose
2}\pi^{2\gamma}\delta^2+\cdots+\pi^{\gamma p}\delta^{p}}{\pi^{\gamma}\delta}\\
&=p+{p\choose 2}\pi^{\gamma}\delta+\cdots \pi^{\gamma (p-1)}\delta^{p-1}\\
&\equiv p \quad ({\rm mod} \ \pi^{e+\gamma}).
\end{align*}

Thus both cases imply  $\val(1 + a_n+ \cdots + a_n^{p-1})=e$.
Since $n>e$, by Lemma \ref{4.6} and Proposition \ref{grow-2}, all the $p^{f-1}$ lifts of $\sigma$ split $e-1$ times then all
the descendants at level $n+e$ grow.
\end{proof}

Now we give a technique lemma which will be useful later.
\begin{lem}\label{order}
Let $\sigma$ be a growing $k$-cycle at level $n\geq e+1$ and set $\gamma=\hat{A}_{n}$.\\
 \indent {\rm 1)} If $\gamma> e/(p-1)$, then $\val(1 + a_n^{p^j}+ \cdots + a_n^{(p-1)p^j})=e $, for all $j\geq 0$.\\
 \indent {\rm 2)} If  $\gamma\leq  e/(p-1)$, then
 \[\val(1 + a_n^{p^j}+ \cdots + a_n^{(p-1)p^j})=\left\{
 \begin{array}{ll}
p^{j} \gamma(p-1)&  \text{ for } \ 0\leq j <\log_{p}\frac{e}{\gamma(p-1)};\\
  e&  \mbox{ for } \  j >\log_{p}\frac{e}{\gamma(p-1)}.
 \end{array}
 \right.
 \]
\end{lem}
\begin{proof}
\indent {\rm 1)}
By the definition of $\gamma$, we have $\gamma\leq n$.

If $\gamma =n$, then $a_n\equiv 1 \ ({\rm mod} \ \pi^{n})$ which implies  $a_n^{p^{j}}\equiv 1 \ ({\rm mod} \ \pi^{n})$.
Thus we have
$ 1+a_n^{p^j}+\cdots+a_n^{(p-1)p^{j}}\equiv p \ ({\rm mod} \ \pi^{n})$. Then by $n\geq e+1$, the assertion follows.

If $e/(p-1)<\gamma< n$, then by putting $a_{n}=1+\pi^{\gamma}\delta_0$ with $|\delta_0|_p=1$, we have
\begin{eqnarray}\label{anj}
a^{p^j}_n=1 + \sum_{i=1}^{p^j} {p^j \choose i}  (\pi^{\gamma}\delta_0)^{i}.
\end{eqnarray}
Since $\gamma> e/(p-1)$,  the equation (\ref{anj}) implies
$$\val(a_n^{p^j}-1)>e/(p-1).$$
Let $\gamma_j=\val(a_n^{p^j}-1)$ and
write $a^{p^j}_n=1+\pi^{\gamma_j}\delta_j$ with $|\delta_j|_p=1$.
Then we have
\begin{align}
1 + a_n^{p^j}+ \cdots + a_n^{(p-1)p^j} &=\frac{a_n^{p^{j+1}}-1}{a_n^{p^j}-1}\\
&=\frac{{p\choose
1}\pi^{\gamma_j}\delta_j+{p\choose
2}\pi^{2\gamma_j}\delta_j^2+\cdots+\pi^{p\gamma_j}\delta_j^p}{\pi^{\gamma_j}\delta_j}  \label{an11} \\
&=p+{p\choose 2}\pi^{\gamma_j}\delta_j+\cdots+ \pi^{(p-1)\gamma_j}\delta_j^{(p-1)}  \label{an12}\\
&\equiv p \quad ({\rm mod} \ \pi^{e+1}).
\end{align}
Thus, 
$$\val(1 + a_n^{p^j}+ \cdots + a_n^{(p-1)p^j})=e. $$

\indent {\rm 2)}  For $0\leq j <\log_{p}(e/((p-1)\gamma))$,  the formula (\ref{anj})  implies
 $$a^{p^j}_n \equiv 1+(\pi^{\gamma}\delta_0)^{p^{j}} \quad ({\rm mod} \ \pi^{e+1}).$$
 So $\val(a^{p^j}_n-1)= p^j\gamma<e/(p-1)$.

Let $\gamma_j=\val(a_n^{p^j}-1)$ and
write $a^{p^j}_n=1+\pi^{\gamma_j}\delta_j$ with $|\delta_j|_p=1$. By (\ref{an11}) and (\ref{an12}), we have
\begin{align*}
1 + a_n^{p^j}+ \cdots + a_n^{(p-1)p^j}&=\frac{a_n^{p^{j+1}}-1}{a_n^{p^j}-1}\\
&\equiv p +\pi^{(p-1)\gamma_j}\delta_j^{(p-1)} \quad ({\rm mod} \ \pi^{e+1}).
\end{align*}
Thus,
$$\val(1 + a_n^{p^j}+ \cdots + a_n^{(p-1)p^j})= \gamma p^{j}(p-1)<e.$$

 For $ j>\log_{2}(e/\gamma)$, by  (\ref{an11}) and (\ref{an12}), we have
\begin{align*}
1 + a_n^{p^j}+ \cdots + a_n^{(p-1)p^j}&=\frac{a_n^{p^{j+1}}-1}{a_n^{p^j}-1}\\
&\equiv p  \quad ({\rm mod} \ \pi^{e+1}).
\end{align*}
So, $$\val(1 + a_n^{p^j}+ \cdots + a_n^{(p-1)p^j})=e. $$
\end{proof}

Lemma \ref{order} gives the following proposition on the behavior of a growing cycle and all of its descendants.
\begin{prop}\label{Cor-grow} Let $\sigma$ be a  growing cycle at level $n$.
 If $n\geq e+1 $ and $\val(a_n-1)>e/(p-1)$, then any lift of $\sigma$ splits $e-1$ times then all the descendants grow; and
again the lifts of the growing cycles split $e-1 $ times then the
descendants grow. This process will go on forever.
\end{prop}

%

Keep the assumption and notation of Proposition \ref{Cor-grow}. If $K$ is a non-trivial extension of $\mathbb{Q}_p$, i.e., $K\neq \mathbb{Q}_p$, then the clopen set
$\mathbb{X}_{\sigma}$ will be decomposed
into infinitely many, more precisely, uncountably many
minimal components.

In fact, if $K\neq \mathbb{Q}_p$, we have $e>1$, or $f>1$. In the case $e>1$, there are always $e-1$ times of splitting during the consecutive growings. Each splitting at some level decomposes each set at this level into $p^f$ parts invariant under $\phi$. Since we have infinitely many times of splitting, we will have uncountably many minimal parts.

In the case of $f>1$, since the lift of a growing cycle has length multiplied by $p$ only, the growing cycle becomes $p^{f-1}$ cycles each time. Thus each set at the growing level is decomposed into $p^{f-1}$ invariant components of $\phi$. By the same reason, we will also have uncountably many minimal parts.

Moreover, we can see each component is not a finite union of balls
any more. They are Cantor type sets. The clopen set $\X_{\sigma}$ with these growing and splitting behaviors is called of
\emph{type} $(k,e)$.

In general, for a given sequence of positive numbers
$\vec{E}=(E_j)_{j\geq 0}$, the clopen set $\mathbb{X}_{\sigma}$ is called of \emph{type} $(k, \vec{E})$
if it is a growing $k$-cycle at level $n$ and the lifts of this $k$-cycle split $E_0-1$
times. Furthermore, all the $E_0$-th generation of descendants grow
and then all the lifts split $E_1-1$ times. Again, the descendants
grow and the lifts of the growing descendants split $E_2-1$ times,
....

Let $(p_s)_{s\geq 1}$ be a sequence of positive integers with the
property that $p_{s}|p_{s+1}$ for all $s\geq 0$. Consider the
inverse limit
$$
      \mathbb{Z}_{(p_s)} : = \lim_{\leftarrow} \mathbb{Z}/p_s
      \mathbb{Z}.
$$
This is a profinite group, usually called an {\em odometer}. The
map $\tau: x \mapsto x+1$ is called the {\em adding machine} on
$\mathbb{Z}_{(p_s)}$.

It can be shown that if $\mathbb{X}_{\sigma}$ is of {type} $(k, \vec{E})$, then $(\mathbb{X}_{\sigma}, \phi)$ can be decomposed into uncountably many minimal subsystems. The power series $\phi$ restricted to each
minimal component is the adding machine with the sequence $(p_s)$
equal to
\[
(p_s)=(k, \ \underbrace{kp, \cdots, kp}_{E_0}, \ \underbrace{kp^2, \cdots,
kp^2}_{E_1}, \ \underbrace{kp^3, \cdots, kp^3}_{E_2}, \ \cdots).
\]

We remark that the type $(k,e)$ is nothing but the type $(k, \vec{E})$ with $\vec{E}=(e,e,\dots)$.

As a consequence of Lemma \ref{order}, we  have the following proposition.
\begin{prop}\label{cor9}
Let $\sigma$ be a growing $k$-cycle at level $n\geq e+1$ and set $\gamma=\hat{A}_n$.\\
\indent {\rm 1)} If $\gamma> e/(p-1)$, then the clopen set
 $\mathbb{X}_\sigma$ is of type $(k, e)$.\\
\indent {\rm 2)} If   $\gamma\leq  e/(p-1)$  and $\log_{p}\frac{e}{\gamma(p-1)}$ is not an integer, then the clopen set
 $\mathbb{X}_{\sigma}$ is of type $(k, \vec{E})$,
with
$$
E_j=\left\{
      \begin{array}{ll}
        p^{j} \gamma(p-1), & \hbox{if $0\leq j<\log_{p}\frac{e}{\gamma(p-1)}$;} \\
        e, & \hbox{if $j>\log_{p}\frac{e}{\gamma(p-1)}$.}
      \end{array}
    \right.$$
\indent {\rm 3)} If $\gamma\leq  e/(p-1)$  and $\log_{p}\frac{e}{\gamma(p-1)}$ is an integer, then all the lifts of $\sigma$
split $\gamma(p-1)-1$ times then all the descendants grow; and
again the lifts of the growing cycles split $\gamma(p-1)p^j-1 $ times then the
descendants grow; ..., the lifts of the growing cycles (at level $n-\gamma+e/(p-1)$ ) split at least $e-1$ times. 
\end{prop}

\bigskip
\section{Minimal decomposition}\label{S-proof}

In this section, we show how we can do minimal decomposition for a convergent series $\phi\in \mathcal{O}_K\langle x\rangle$.

If a cycle grows tails, it will produce  an attracting periodic orbit with an
attracting basin. If a splitting cycle always has a splitting lift with the same length,
then it will produce a periodic orbit of $\phi\in \mathcal{O}_K\langle x\rangle.$
For a growing  cycle $\sigma$ at level $n\geq e+1$, if  $\hat{A}_n>e/(p-1)$ and $K\neq \Qp$, then it will produce uncountably many  minimal
components of $\phi$.

We shall describe these assertions more precisely.

Let $\sigma=(x_1,\dots,x_k)$ be a cycle of $\phi_n$. Recall that in
this case $\sigma $ is called a $k$-cycle at level $n$.
There are  four special situations for the dynamical system $\phi :
\mathbb{X}_\sigma \to \mathbb{X}_\sigma$.\\

\indent (S1) Suppose $\sigma $ grows tails. Then $\phi$ admits a
$k$-periodic orbit with one periodic point in each ball
$\mathbb{X}_i=\{x \in \O: x\equiv x_i \ ({\rm mod} \ \pi^n) \} (1\leq i \leq k)$, and all other points in
$\mathbb{X}_\sigma$ are attracted into this orbit. In this situation, if
$x$ is a point in the $k$-periodic orbit, then $|(\phi^{\circ k})'(x)|_p<1 $
since $(\phi^{\circ k})'(x)=a_m(x) \equiv 0 \ ({\rm mod} \ \pi) $ for all
$m\geq n $. The periodic orbit $(x, \phi(x), \cdots, \phi^{\circ(k-1)}(x))$ is
then attractive.\\

\indent (S2) Suppose $\sigma$ grows at level $n\geq e+1 $ and  $\hat{A}_n>e/(p-1)$. By Proposition \ref{Cor-grow},  any lift of the  growing cycle $\sigma$, splits $e-1$ times then all the descendants grow; and
again the lifts of the growing cycles split $e-1 $ times then the
descendants grow. This process will go on forever. The dynamical system $(\mathbb{X}_\sigma,\phi)$  which consists of uncountably many minimal subsystem, is of type $(k,e)$.\\

\indent (S3) Suppose $\sigma$ splits and there is a splitting lift with the same length as $\sigma$
at level $n+1$, and then for this lift there is still a splitting lift with the same length, and go on. Then there is a $k$-periodic orbit
with one periodic point in each $\mathbb{X}_i (1\leq i \leq
k)$. We say that $\sigma $ is a starting splitting cycle at level
$n$. In this situation, if $x$ is a point in the $k$-periodic orbit,
then $|(\phi^{\circ k})'(x)|_p=1 $ since $(\phi^{\circ k})'(x)=a_m(x) \equiv 1 \ ({\rm mod} \
\pi) $ for all $m\geq n $. Thus the periodic orbit $(x, \phi(x),
\cdots, \phi^{\circ (k-1)}(x))$ is indifferent.\\

\indent (S4) Suppose  $\sigma$ partially splits.  Then by Lemma \ref{anbn}, there is
one lift of length $k$ which still partially splits like $\sigma $.
Thus there is a $k$-periodic orbit with one periodic point in each
$\mathbb{X}_i\ (1\leq i \leq k)$. In this situation, if $x$
is a point in the $k$-periodic orbit formed above, then
$|(\phi^{\circ k})'(x)|_p=1 $ since $(\phi^{\circ k})'(x)=a_m(x) \not\equiv 0,1 \  ({\rm
mod} \ \pi) $ for all $m\geq n $. Hence, the periodic orbit $(x,
\phi(x), \cdots, \phi^{\circ (k-1)}(x))$ is  indifferent.\\


Before proving Theorem \ref{thm-decomposition},  we first show that there are only finitely many possible periods of periodic orbits.
Let $\sigma=(x_1,\dots,x_k)$ be a    cycle  at level $n$. Recall that  $\hat{A}_{n}(x)=\min\{A_{n}(x),n\}$ does not depend on the choice of $x\in \mathbb{X}_\sigma $. So, without misunderstanding, we will not mention $x$ in $\hat{A}_n(x)$.
\begin{prop}\label{cyclelengthgrow}
Let $\sigma=(x_1,\dots,x_k)$ be a  growing  cycle  at level $n\geq e+1$. \\
 \indent {\rm 1)} If $\log_p \frac{e}{(p-1)\hat{A}_n}$ is  not a nonnegative  integer, then there is no periodic point in $\mathbb{X}_\sigma.$\\
\indent {\rm 2)} If $\log_p \frac{e}{(p-1)\hat{A}_n}$ is   a nonnegative  integer,
 then the possible  periods of the periodic orbits in $\mathbb{X}_\sigma$ must be  
 $\frac{kp e}{(p-1)\hat{A}_n}$.
\end{prop}
\begin{proof}
 {\rm 1)} The conclusion  is  a direct consequence of the assertions {\rm 1)} and {\rm 2)} of Proposition \ref{cor9}.
 
 {\rm 2)} Let $\sigma_s$ be a cycle which is a descendant of $\sigma$
at level $s:=n-\hat{A}_n+ep/(p-1)$. By the third assertion of Proposition \ref{cor9}, $\sigma_s$ is a growing or splitting cycle of length $\frac{kpe}{(p-1)\hat{A}_n}$ with
$\hat{A}_s\geq \min\{ep/(p-1),n\}>e.$

If  the  cycle $\sigma_s$ grows, then there is no periodic point in the clopen set $\mathbb{X}_{\sigma_{s}} =\bigsqcup_{x\in \sigma}(x +\pi^s \O)$.

If the  cycle $\sigma_s$ splits, we first suppose that there is a splitting lift with the same length as $\sigma$ at level $n+1$, and then for this lift there is still a splitting lift with the same length, and go on.  Then there is a $\frac{kpe}{(p-1)\hat{A}_n} $-periodic orbit
with one periodic point in each $x+\pi^s\O \ (x\in\sigma_s)$. If it is not the case, let $\sigma_t$ be a growing  descendant of $\sigma_s$  at level $t>s$  with the same length as $\sigma_s$. Since $\hat{A}_s>e$, it follows  $\hat{A}_t\geq e$. So
the clopen set $\mathbb{X}_{\sigma_t}=\bigsqcup_{x\in \sigma_t} (x +\pi^t \O)$ is of type $(\frac{kpe}{(p-1)\hat{A}_n},e)$. Thus there is no periodic point in $\mathbb{X}_{\sigma_t}$.

\end{proof}

\begin{prop}\label{cyclelengthsplit}
Let $\sigma=(x_1,\dots,x_k)$ be a  splitting  cycle  at level $n\geq e+1$. Then the possible  periods of periodic orbits in $\mathbb{X}_\sigma$ must be $k$ or
$\frac{kp e}{(p-1)\hat{A}_n}$ (when $\log_p \frac{e}{(p-1)\hat{A}_n}$ is  a nonnegative  integer).

\end{prop}
\begin{proof}
Suppose first that there is a splitting lift with the same length as $\sigma$ at level $n+1$, and then for this lift there is still a splitting lift with the same length, and go on. Then there is a $k$-periodic orbit
with one periodic point in each $x_i+\pi^n\O \ (x_i\in\sigma)$.

If not, let $\sigma_s$ be a growing  descendant of $\sigma$  at level $s>n$  with the same length to $\sigma$. As in Proposition \ref{cyclelengthgrow}, we have two cases.

\indent {\rm 1)} If  $\log_p \frac{e}{(p-1)\hat{A}_n}$ is  not a nonnegative  integer, then neither is $\log_p \frac{e}{(p-1)\hat{A}_s}$. So there is no periodic point in the clopen set $\mathbb{X}_{\sigma_s} =\bigsqcup_{x\in \sigma_s}(x +\pi^s \O)$.\\
\indent {\rm 2)} If  $\log_p \frac{e}{(p-1)\hat{A}_n}$ is  a nonnegative  integer, then  so is $\log_p \frac{e}{(p-1)\hat{A}_s}$. By  Proposition \ref{cyclelengthgrow},  the possible  periods of periodic orbits in $\mathbb{X}_{\sigma_s}$ must be   $\frac{kp e}{(p-1)\hat{A}_n}$.
\end{proof}
An immediate  consequence of Propositions \ref{cyclelengthgrow} and \ref{cyclelengthsplit} is the following corollary.
\begin{cor}\label{possible-periods}
For each $\phi \in \mathcal{O}_K\langle x\rangle$, there are only finitely many possible periods of periodic orbits in $\O$.
\end{cor}
Now we can prove our main theorem.
\begin{proof}[Proof of Theorem \ref{thm-decomposition}]
Let us first prove that there are only finitely many periodic points.   In fact, there are only finitely many possible lengths of periods by Corollary \ref{possible-periods}.
But periodic points are solutions
of the equations $\phi^{\circ q_i}(x) = x$ with $\{q_i\}$ being one of the finite possible lengths of periods. Since $\mathcal{O}_K\langle x\rangle$ is closed under composition of functions, it follows  $\phi^{\circ q_i} \in \mathcal{O}_K\langle x\rangle$.  Write \[\phi^{\circ q_i}(x)-x= \sum_{i\geq 0}c_i x^i\] and let $j$  be the largest integer such that $|c_j|_p= \max_{i\geq 0}|c_i|_p.$
By Proposition \ref{WPT}, there is a monic polynomial $g\in \O[x]$ of degree $j$ and a power series $h\in\mathcal{O}_K\langle x\rangle$ such that $\phi^{\circ q_i}(x)-x=g(x)h(x)$, and $h(x)\neq 0 $ for all $x\in \O$.  Since $g$ is a polynomial, the  equation $g(x)=0$ admits only a finite number of solutions. Hence, there are only a finite number of
periodic points.

Let us continue the proof.
We start from the level $e+1$. The space $\O$ is a union of $p^{(e+1)f}$
balls with radius $p^{-(e+1)/e}$. Each ball is identified with a point in
$\O/\pi^{e+1}\O$. The induced map $\phi_{e+1}$ on $\O/\pi^{e+1}\O$ admits some
cycles. The points outside the cycles are mapped into the cycles. The balls corresponding to the points outside the cycles will be put into the third part
$C$ in the decomposition. From now on, we will be concerned only with cycles at level
$n\ge e+1$.

Let  $\sigma=(x_1,\dots,x_k)$ be a cycle at level $n \geq
e+1$.
We distinguish four cases.\\
 \indent (P1) {\em $\sigma$ grows tails.}  Then by (S1), the clopen set $\mathbb{X}_{\sigma}
$ consists of a $k$-periodic orbit and other points are attracted by
this periodic orbit. So, $\mathbb{X}_{\sigma}$ contributes to the first part
$A$ and the third part $C$ in the decomposition.

\indent (P2) {\em $\sigma $ grows}.  We shall apply Lemma \ref{order} and Proposition \ref{cor9}.
\begin{itemize}
  \item[(1)] If $\hat{A}_n> e/(p-1)$, then the clopen set $\mathbb{X}_{\sigma}$ is of type $(k,e)$. So, $\mathbb{X}_{\sigma}\subset B$ consists some minimal components of the part $B$ in the decomposition.
  \item[(2)] If $\hat{A}_n \leq e/(p-1)$ and $\log_p\frac{p}{(p-1)\hat{A}_n}$ is not an integer,  then the clopen set
 $\mathbb{X}_{\sigma}
=\bigsqcup_{i=1}^k (x_i +\pi^n \O)$ is of type $(k, \vec{E})$,
with
$$
E_j=\left\{
      \begin{array}{ll}
        p^{j} (p-1)\hat{A}_n, & \hbox{if $0\leq j<\log_{p}\frac{e}{(p-1)\hat{A}_n}$;} \\
        e, & \hbox{if $j>\log_{p}\frac{e}{(p-1)\hat{A}_n}$.}
      \end{array}
    \right.$$
Consider the level $s:=n+\hat{A}_n p^{\lfloor \log_{p}\frac{e}{(p-1)\hat{A}_n}\rfloor+1}-\hat{A}_n.$ Let $\sigma_s$
be any cycle at level $s$ which is a descendant of $\sigma$ at level $n$. Then $\sigma_s$ is a growing cycle of length $\ell:=kp^{\lfloor \log_{p}\frac{e}{(p-1)\hat{A}_n}\rfloor+1}$. The clopen set  $\mathbb{X}_{\sigma_s}
=\bigsqcup_{x\in \sigma_s} (x +\pi^s \O)$ is then of type $(\ell, e)$. So  $\mathbb{X}_{\sigma_s}\subset B$ and $\mathbb{X}_\sigma \subset B$.
  \item[(3)]  If $\hat{A}_n \leq e/(p-1)$ and $\log_p\frac{p}{(p-1)\hat{A}_n}$ is an integer, then we need to study the descendants of $\sigma$ at level $s=n-\hat{A}_n+ep/(p-1)$ separately. Let $\sigma_s$
be such a descendant. Then $\sigma_s$ is a growing or splitting cycle of length $\ell:=\frac{kpe}{(p-1)\hat{A}_n}$ with $\hat{A}_s\geq \min\{ep/(p-1),n\}>e$. If $\sigma_s$ grows, then we go to the case of (1). If $\sigma_s$ splits, then we go to  (P3).
\end{itemize}

\indent (P3) {\em  $\sigma$ splits}.  We shall apply
Proposition \ref{splits-p3}.
\begin{itemize}
  \item[(1)] If $\sigma$ belongs to Case 1 of Proposition
\ref{splits-p3}, then after finitely many times of splitting, all
lifts grow. So all the lifts are in some case of (P2) determined by $\hat{A}_n$.

\item[(2)] If $\sigma$ belongs to Case 2 of Proposition
\ref{splits-p3}, then there is one
lift of $\sigma$ sharing the property (S3), and there is a periodic orbit with period $k$. Each lift
different from the cycle deriving the periodic orbit (at any level
$m \geq n+1 $)  find itself  in some case of (P2) (determined by $\hat{A}_n)$ after finitely many times of lifting.

\item[(3)] If $\sigma$ belongs to
Case 3 of Proposition
\ref{splits-p3}, then $\sigma $ splits into $p^{nf} $ cycles at level $2n$.
These cycles at level $2n$ may continue this procedure of analysis
of (P3). But this procedure can not continue infinitely, because
there are only a finite number of periodic points. So, all these
cycles may continue to split but they must end with their lifts
belonging either to Case 1 or Case 2 in Proposition \ref{splits-p3}.
So, $\mathbb{X}_\sigma$ contributes to both $A$ and $B$.
\end{itemize}

\indent (P4) {\em  $\sigma$ partially splits}. In this case, $\sigma$ is in
the situation (S4).  Thus there comes out a periodic orbit. For $m \geq n $,  we suppose
that
$\sigma_{m+1} $ is the lift of $\sigma$ at
level $m+1$ deriving the periodic orbit.  Then the other lifts different from
$\sigma_{m+1}$ will grow or split. So we  go to (P2) or (P3) for these cycles
at level $m+1$.

For the case (3) of (P2), a splitting  lift $\sigma_s$ of $\sigma$ at level $s=n-\hat{A}_n+ep/(p-1)$ must have $\hat{A}_s>e$. By the process of  (P3),  the growing  descendants of $\sigma_s$
 are in the case (1) of (P2), so the procedures will stop.
    Thus, all the above  procedures will stop and  we get the decomposition.
\end{proof}

\bigskip
\section{Affine polynomial dynamics}\label{S-affine}

  In this section, we would like to study the minimal
decomposition of the affine polynomial dynamical system $(K, F)$ with $F(x)=\a x+\b, \ (\a, \b\in K, \a\neq 0, (\a,\b)\neq (1,0))$ by using  the idea of cycle lifting.

For $F(x)=\b$ or $F(x)=x$, the dynamical system $(K,F)$ is trivial.

We distinguish two cases: $F$ is a translation or not.\\
\textbf{Case I}.  If $F(x)=x+\b$ with $\b\neq 0$, then  $F$ is conjugate to the translation $\hat{F}(x)=x+1$ by $h(x)=x/\b$.

We consider $\O$ as a single point cycle at level $0$, denoted by $(0)$. We say that the cycle $(0)$ at level $0$ grows if $\O/\pi \O$ consists of $p^{f-1}$ cycles of length $p$ under $F_1$. Then $(\O,F)$ is said to be of type $(1,e)$ at level $0$ if the cycle $(0)$ grows at level $0$, the $p^{f-1}$ lifts of $(0)$  with length $p$ at level $1$ splits $e-1$ times then all the descendants grow; and
again the lifts of the growing cycles split $e-1 $ times then the
descendants grow; and so on.

For $a\in K$, denote by $\overline{\B}(a,1)=\{x\in K, |x-a|_p\leq 1\}$ the ball centered at $a$ with radius $1$. We have the following theorem.

\begin{thm}For the translation $F(x)=x+1$ acting on $K$, each ball of radius $1$ is $F$-invariant.
For $a\in K$, each subsystem $(\overline{\B}(a,1),F)$ is conjugate to $(\O,F)$ by $h(x)=x-a$. The dynamical system $(\O,F)$ is of type $(1,e)$ at level $0$.
\end{thm}
\noindent\textbf{Case II}. If $F(x)=\a x+\b,~~\a\neq 1$, then $F$ is conjugate to the multiplication $\hat{F}(x)=\a x $ by $h(x)=x-(\b/(1-\a))$.
Thus we need only to study the multiplication dynamics $(K,F(x)=\a x)$.
It is easy to see that $0,\infty$ are  two fixed points of $F$.

Let $\U:=\{x\in \O:
|x|_p=1\} $ be the group of units in $\O$ and $\V:=\{x\in \U: \exists m \in \N, m\geq 1, x^m=1\}$ be the set of roots of unity in $\O$. The set $\V$ is a finite set of  $(p^f-1)p^s$ elements, with $p^s$ being the highest power of $p$ such that $K$ has a root of unity of order $p^s$. For the details, see the book of Robert \cite{Rob-GTM198}.
We will distinguish three cases:
\[
\rm{(A)} \  \a\not \in \U, \quad
\rm{(B)} \ \a\in  \V \setminus \{1\}, \quad \text{and}\quad
\rm{(C)} \ \a\in \U \setminus \V. \]

The following we will treat the three cases separately. The first two cases are easy.

Case (A) $\a\not \in \U$. If $|\a|_p<1$, then the system admits an attracting fixed point $0$ with the whole $K$ being the attracting basin, that is,
$$\lim_{n\rightarrow \infty } F^{\circ n}(x) = 0,~~~~~~~~~\forall x\in K .$$
 If $|\a|_p>1$, then the system admits a repelling fixed point $0$ and the whole $K$ except $0$ lies in  the attracting basin of $\infty$, that is,
$$\lim_{n\rightarrow \infty } F^{\circ n}(x) = \infty,~~~~~~~~~\forall x\in K \setminus \{0\}.$$

Case (B) $\a\in  \V \setminus \{1\}$. Let $\ell$ be the order of $\a$, i.e., the least integer such that $\a^{\ell}=1 $. It is easy to see that all points are in a periodic orbit with period $\ell$.

Case (C) $\a\in \U \setminus \V$. For each $n\in \Z$, the sphere $\pi^{-n}\U$  is $F$-invariant.  The dynamical system $(\pi^{-n}\U,F)$ is conjugate to the system $(\U,F)$ by $h(x)=\pi^{n}x$.
\begin{center}
\begin{picture}(100,90)(0,-65) \put(-5,0){$\pi^{-n}\U$}
\put(25,6){\vector(1,0){40}} \put(70,0){$\pi^{-n}\U$} \put(37,10){$\alpha x
$} \put(10,-10){\vector(0,-1){30}} \put(-10,-25){$\pi^{n} x$}
\put(79,-10){\vector(0,-1){30}} \put(82,-25){$\pi^{n} x $}
\put(8,-55){$\U$} \put(75,-55){$\U$}
\put(25,-52){\vector(1,0){40}} \put(37,-63){$\alpha x$}
\end{picture}
\end{center}

Now we are going to study the dynamical system $(\U,F)$.

\begin{thm}
Consider the system $(\U, F)$ where $F(x)=\a x$ with $\a\in \U \setminus \V$. Let $\ell$ be the order of $\alpha$ in $\mathbb{K}^{*}$. One can decompose $\mathbb{U}$ into
   $(p^f-1) p^{v_{\pi}(\a^\ell-1)\cdot f-f}/\ell$ clopen sets such that each
  clopen set is of type $
(\ell, \vec{E})$ at level $v_{\pi}(\a^\ell-1)$ where
$$\vec{E}=\left(\val(\frac{\a^{\ell p }-1}{\a^{\ell}-1}), \ \val(\frac{\a^{\ell p^2 }-1}{\a^{\ell p}-1}), \ \cdots, \ \val(\frac{\a^{\ell p^{N+1} }-1}{\a^{\ell p^N}-1}), \ e, \ e, \ \cdots\right).$$
Here $N$ is the largest integer such that $\val((\a^{\ell p^{N+1} }-1)/({\a^{\ell p^N}-1}))\neq  e$.

\end{thm}
\begin{proof}
Notice that  $\ell $ is  the order of $\a$ in $\K^* $. Thus $\a^{\ell} \equiv 1 \ ({\rm mod} \ \pi)$.
  We can check that there are  $(p^f-1)/\ell $ cycles of length $\ell
  $ at level $1$.
Now we consider the $\ell$-cycles at level $1$.

For each $x_0\in \U$,
  \begin{align*}
a_1(x_0)&=(F^{\circ \ell})'(x_0)= \a^\ell, \\
b_1(x_0)&=\frac{F^{\circ \ell}(x_0)-x_0}{\pi}=\frac{(\a^\ell-1)x_0}{\pi}.
  \end{align*}
Thus $a_1(x_0)\equiv 1 \ ({\rm mod} \ \pi)$ and
  $\val(b_{1}(x_0))=\val(\a^\ell-1)-1.$
By induction, one can check that each $\ell$-cycle at level $1$ splits $\val(\a^\ell-1)-1 $ times and then
all its lifts grow.

Let $m=\val(\a^\ell-1)+1 $. Consider the $\ell p$-cycles at level
$m$.
 For each $x_0\in \U$,
  \begin{align*}
a_m(x_0)&=(F^{\circ \ell p})'(x_0)= \a^{\ell p}, \\
b_m(x_0)&=\frac{F^{\circ \ell p}(x_0)-x_0}{\pi^m}=\frac{(\a^{\ell
p}-1)x_0}{\pi^m}.
  \end{align*}
Then we have
\[
 \val(b_m(x_0))=\val({\a^{\ell p}-1})-m=\val(\frac{\a^{\ell p }-1}{\a^{\ell}-1})-1.
\]
Hence each $\ell p$-cycle splits  $\val(\frac{\a^{\ell p }-1}{\a^{\ell}-1})-1$ times and then all its lifts
grow.

Now let $q=\val(\a^{\ell p}-1)+1 $. Consider the $\ell p^2$-cycles at level
$q$. By the same calculations, we have for each $x_0$,
$$a_q(x_0)= \a^{\ell p^2}\equiv 1 \ ({\rm mod} \ \pi)$$ and
\[
 \val(b_q(x_0))=\val({\a^{\ell p^2}-1})-q=\val(\frac{\a^{\ell p^2 }-1}{\a^{\ell p}-1})-1.
\]
Hence each $\ell p^2$-cycle splits $\val(\frac{\a^{\ell p^2 }-1}{\a^{\ell p}-1})-1$ times and then all its lifts
grow.

Go on this process, we can show that each $\ell p^k$-cycle  at level $\val(\a^{\ell p^{k-1}}-1)+1 $ splits $\val(\frac{\a^{\ell p^k }-1}{\a^{\ell p^{k-1}}-1})-1 $ times and then all its lifts
grow.

Since $\val(\a^{\ell i p^k} -1) \to \infty$ when $k\to \infty$ for $1\leq i \leq p-1$ (see \cite{sch}, p.100), and $\val(p)=e$, we can find
an integer $N$ such that for all $k> N$, we have
\begin{align*}
\val(\frac{\a^{\ell p^{k+1} }-1}{\a^{\ell p^k}-1})-1&=\val(1+\a^{\ell p^k}+\cdots+\a^{(p-1)\ell p^{k}})-1\\
&=\val((1-1)+(\a^{\ell p^k}-1)+\cdots+(\a^{(p-1)\ell p^{k}}-1)+p)-1\\
&=\val(p)-1=e-1.
\end{align*}
Therefore, we can conclude that each $\ell$-cycle at level $\val(\a^\ell-1)$ is of type $
(\ell, \vec{E})$ with

$$\vec{E}=\left(\val(\frac{\a^{\ell p }-1}{\a^{\ell}-1}), \ \val(\frac{\a^{\ell p^2 }-1}{\a^{\ell p}-1}), \ \cdots, \ \val(\frac{\a^{\ell p^{N+1} }-1}{\a^{\ell p^N}-1}), \ e, \ e, \ \cdots\right).$$

On the other hand, we have that  the number of $\ell$-cycle at level $\val(\a^\ell-1)$ is
$$\frac{p^f-1}{\ell} \cdot (p^{f})^{\val(\a^\ell-1)-1},$$ since one $\ell$-cycle splits into $p^f$ number of $\ell$-cycles. This completes the proof.

\end{proof}

\section*{Acknowledgement}
The authors thank Professor Ai-Hua Fan for his valuable remarks. Shi-Lei Fan was partially supported by NSF of China (Grant No. 11231009) and CNRS program (PICS No.5727), Lingmin Liao was partially supported by 12R03191A - MUTADIS and the project PHC Orchid of MAE and MESR of France.

%
%
%
%
%


\begin{thebibliography}{10}

\bibitem{Ana94}
V.~S. Anashin.
\newblock Uniformly distributed sequences of {$p$}-adic integers.
\newblock {\em Mat. Zametki}, 55(2):3--46, 188, 1994.

\bibitem{Ana02}
V.~S. Anashin.
\newblock Uniformly distributed sequences of {$p$}-adic integers, {II}.
\newblock {\em Diskret. Mat.}, 14(4):3--64, 2002.

\bibitem{Ana06}
V.~S. Anashin.
\newblock Ergodic transformations in the space of {$p$}-adic integers.
\newblock In {\em {$p$}-adic mathematical physics}, volume 826 of {\em AIP
  Conf. Proc.}, pages 3--24. Amer. Inst. Phys., Melville, NY, 2006.

\bibitem{AK09AAD}
V.~S. Anashin and A.~Khrennikov.
\newblock {\em Applied algebraic dynamics}, volume~49 of {\em de Gruyter
  Expositions in Mathematics}.
\newblock Walter de Gruyter \& Co., Berlin, 2009.

\bibitem{AKY11}
V.~S. Anashin, A.~Khrennikov, and E.~I. Yurova.
\newblock Characterization of ergodic {$p$}-adic dynamical systems in terms of
  the van der {P}ut basis.
\newblock {\em Dokl. Akad. Nauk}, 438(2):151--153, 2011.

\bibitem{BGR}
S.~Bosch, U.~G{\"u}ntzer, and R.~Remmert.
\newblock {\em Non-{A}rchimedean analysis}, volume 261 of {\em Grundlehren der
  Mathematischen Wissenschaften [Fundamental Principles of Mathematical
  Sciences]}.
\newblock Springer-Verlag, Berlin, 1984.
\newblock A systematic approach to rigid analytic geometry.

\bibitem{CFF09}
J.~Chabert, A.~H. Fan, and Y.~Fares.
\newblock Minimal dynamical systems on a discrete valuation domain.
\newblock {\em Discrete Contin. Dyn. Syst.}, 25(3):777--795, 2009.

\bibitem{CP11Ergodic}
Z.~Coelho and W.~Parry.
\newblock Ergodicity of {$p$}-adic multiplications and the distribution of
  {F}ibonacci numbers.
\newblock In {\em Topology, ergodic theory, real algebraic geometry}, volume
  202 of {\em Amer. Math. Soc. Transl. Ser. 2}, pages 51--70. Amer. Math. Soc.,
  Providence, RI, 2001.

\bibitem{DZunpu}
D.~L. desJardins and M.E. Zieve.
\newblock Polynomial mappings mod $p^n$.
\newblock arXiv:math/0103046v1.

\bibitem{DP09}
F.~Durand and F.~Paccaut.
\newblock Minimal polynomial dynamics on the set of 3-adic integers.
\newblock {\em Bull. Lond. Math. Soc.}, 41(2):302--314, 2009.

\bibitem{FLYZ07}
A.~H. Fan, M.~T. Li, J.~Y. Yao, and D.~Zhou.
\newblock Strict ergodicity of affine {$p$}-adic dynamical systems on {$\Bbb
  Z_p$}.
\newblock {\em Adv. Math.}, 214(2):666--700, 2007.

\bibitem{FL11}
A.~H. Fan and L.~M. Liao.
\newblock On minimal decomposition of p-adic polynomial dynamical systems.
\newblock {\em Adv. Math.}, 228:2116--2144, 2011.

\bibitem{FFLW2013}
A.~H. Fan, S.~L. Fan, L.~M. Liao, and Y.~F. Wang.
\newblock On minimal decomposition of $p$-adic homographic dynamical systems. \newblock {\em Adv. Math.},  257:92--135, 2014.

\bibitem{GKL01}
M.~Gundlach, A.~Khrennikov, and K.~Lindahl.
\newblock On ergodic behavior of {$p$}-adic dynamical systems.
\newblock {\em Infin. Dimens. Anal. Quantum Probab. Relat. Top.},
  4(4):569--577, 2001.

\bibitem{Jeong2013}
S.~Jeong.
\newblock Toward the ergodicity of {$p$}-adic 1-{L}ipschitz functions
  represented by the van der {P}ut series.
\newblock {\em J. Number Theory}, 133(9):2874--2891, 2013.

\bibitem{KN04book}
A.~Khrennikov and M.~Nilson.
\newblock {\em {$p$}-adic deterministic and random dynamics}, volume 574 of
  {\em Mathematics and its Applications}.
\newblock Kluwer Academic Publishers, Dordrecht, 2004.

\bibitem{KobGTM58}
N.~Koblitz.
\newblock {\em {$p$}-adic numbers, {$p$}-adic analysis, and zeta-functions},
  volume~58 of {\em Graduate Texts in Mathematics}.
\newblock Springer-Verlag, New York, second edition, 1984.

\bibitem{Larin02}
M.~V. Larin.
\newblock Transitive polynomial transformations of residue rings.
\newblock {\em Diskret. Mat.}, 14(2):20--32, 2002.

\bibitem{Mah81}
K.~Mahler.
\newblock {\em {$p$}-adic numbers and their functions}, volume~76 of {\em
  Cambridge Tracts in Mathematics}.
\newblock Cambridge University Press, Cambridge, second edition, 1981.

\bibitem{Rob-GTM198}
A.~Robert.
\newblock {\em A course in {$p$}-adic analysis}, volume 198 of {\em Graduate
  Texts in Mathematics}.
\newblock Springer-Verlag, New York, 2000.

\bibitem{sch}
W.~H. Schikhof.
\newblock {\em Ultrametric calculus}, volume~4 of {\em Cambridge Studies in
  Advanced Mathematics}.
\newblock Cambridge University Press, Cambridge, 2006.
\newblock An introduction to $p$-adic analysis, Reprint of the 1984 original
  [MR0791759].

\bibitem{SilvermanGTM241}
J.~Silverman.
\newblock {\em The arithmetic of dynamical systems}, volume 241 of {\em
  Graduate Texts in Mathematics}.
\newblock Springer, New York, 2007.

\end{thebibliography}
\end{document}